\documentclass[12pt,a4paper,oneside,onecolumn]{article}
\usepackage{mathrsfs}
\usepackage{amsfonts}
\usepackage{latexsym}
\usepackage{amsmath,amsthm}
\usepackage{amssymb}
\usepackage{epsf}
\usepackage{graphicx}
\usepackage{indentfirst}
\usepackage{cite}
\usepackage{color}
\usepackage{diagbox}
\usepackage{booktabs}
\usepackage{multirow}
\usepackage{caption}

\theoremstyle{plain}
\newtheorem{thm}{\bf Theorem}[section]

\newtheorem{cor}[thm]{\bf Corollary}

\newtheorem{lem}[thm]{\bf Lemma}

\newcommand{\supp}{\mathrm{supp}}

\theoremstyle{definition}
\newtheorem{defi}{\bf Definition}
\newtheorem{exam}{\bf Example}

\theoremstyle{remark}
\newtheorem{rmk}{\bf Remark}

\title{\bf  Improved Bounds for \\Nested Orthogonal Arrays}
\author{Xiaodong Niu$^{1}$,  Guangzhou Chen$^{2}$, Zihong Tian$^{1}$, Jianguo Lei$^{1*}$\\
\small1. {\it School of Mathematical Sciences, Hebei Normal University, Shijiazhuang, 050024, P.R.China}\\
\small 2.  {\it School of Mathematics and Statistics, Henan Normal University, Xinxiang, 453007, P.R.China }\\
}
\date{}
\begin{document}
\maketitle

\begin{center}
\center  {\bf Abstract}
\begin{minipage}{16cm}
{

\vspace{0.3cm} \ \ \ \  Nested orthogonal arrays (NOAs) have found increasing application in various experimental design problems. A central challenge in this field is the derivation of lower bounds on the number of runs. These bounds serve as a powerful criterion to prove the nonexistence of specific arrays. For symmetric NOAs, Mukerjee, Qian, and Wu developed statistical arguments that yield fairly tight bounds. By contrast, the bounds for asymmetric NOAs proposed by Lin, Pang and Chen, which are obtained via a recursive column deletion technique that reduces the general problem to NOAs of strength 2, are not optimal. Consequently, improving these bounds remains a significant open problem.

\quad \quad In this paper, we reformulate the Rao bound and establish a new bound for NOAs under a group-theoretic framework. Using the character theory of finite abelian groups, we obtain an equivalent characterization via group characters. This framework allows us to give a new proof of Rao's bound for orthogonal arrays and to derive significantly sharper lower bounds for asymmetric NOAs than those of Lin et al. In the case where all factor levels are equal, our bounds reduce naturally to the symmetric bounds of Mukerjee, Qian, and Wu. We also confirm their optimality by explicitly two  constructions of NOAs that achieve these bounds.\vskip 0.1cm


}
\end{minipage}
\end{center}

\vskip 0.5cm

 {  \begingroup\makeatletter  \let\@makefnmark\relax  \footnotetext{%
  \scriptsize\parbox{\linewidth}{%
    * Corresponding author: leijg1964@hebtu.edu.cn (J. Lei)
}%
}\makeatother\endgroup}

\section{Introduction}

Orthogonal array is a combinatorial configuration introduced firstly by Rao \cite{Rao1947} in 1947. It is one of important topics in combinatorial design theory \cite{Colbourn2007,Stinson2004} and experimental design theory \cite{Dey,Fang2006,Taguchi1987}. Orthogonal arrays have found widespread applications in various fields such as fractional factorial experiments\cite{Hedayat1999,Cheng1980,Mukerjee}, computer experiments\cite{Owen}, survey sampling\cite{McCarthy}, computer science\cite{Fang2006,lin10,Williams2001}, coding theory\cite{Bierbrauer2005,Hedayat1999}, cryptography\cite{Carlet,Dong-Pei2008,Pei2009}, quantum information\cite{Goyeneche2016}, and wireless sensor networks\cite{Du}.
Orthogonal arrays of strength $2$ have been extensively studied, constructions and related results can be found in \cite{Bose1952,Chen,Hedayat1999,Colbourn2007}. For orthogonal arrays of strength 3 and higher, we refer the reader to \cite{Brouwer2006,Cao2024,chen17,chen23,Hedayat96,Ji10,Nguyen,Niu2025,Pang,yin11}.

In many experimental contexts, a single orthogonal array is insufficient, since sequential experimental stages often differ in precision, computational cost, and resource requirements. Orthogonal arrays, known for space-filling and projection properties, serve as standard tools for physical experiments and multi-fidelity numerical simulations. These practical demands motivate the development of \emph{nested orthogonal arrays}, introduced by Mukerjee et al. \cite{MQW} as a natural generalization of classical orthogonal arrays.

A nested orthogonal array (NOA) is a hierarchical combinatorial design consisting of two orthogonal arrays, in which a smaller, high-precision subarray is embedded within a larger, low-cost parent array. This nested structure enables economical collaborative sampling across multiple accuracy levels \cite{Pang2024}. NOAs are particularly valuable for computer experiments \cite{Dey}, where multi-fidelity simulations with varying in both accuracy and computational costs have become increasingly essential\cite{Qian2009}, driven by the prohibitive expense of physical experiments\cite{Qian2009-Ann}. By providing a robust framework for constructing nested space-filling designs, NOAs allow high-accuracy experiments to be nested within larger, more affordable ones, facilitating substantial progress in the modeling and analysis of such hierarchical data \cite{Kennedy2000,Qian2006,Qian2006-Bi,Pang2025}.

Bounds for symmetric nested orthogonal arrays were established by Mukerjee et al. \cite{MQW}, while Lin et al.\cite{Lin} provided bounds for the asymmetric case.
Since then, an increasing number of scholars have turned their attention to the existence of NOAs and have studied NOAs. Based on finite field theory, Dey \cite{Dey1,Dey2} developed a series of systematic construction methods and enriched the family of feasible symmetric NOAs.  Qian and Ai \cite{Qian2010-JASA} constructed symmetric NOAs using nested difference matrices, incomplete pairwise orthogonal Latin squares, and the Bush construction. Wang and Li, together with Wang and Yin \cite{Wang-Li2013,Wang-Yin2013}, presented symmetric NOAs with fewer than seven levels via classical methods of combinatorial design. Qian et al. \cite{Qian2014,Qian2009-Ann,Qian2009} introduced level-folding and replacement techniques to generate asymmetric NOAs of strength $2$ with prime power levels. Zhang et al. \cite{Zhang2018,Zhang2019} employed generator matrices to construct several new infinite families of NOAs. Recently, Pang and Zhu \cite{Pang2024,Pang2024-SPL} used orthogonal Latin squares to construct a series of new infinite families of NOAs with adjacent levels. Pang et al. \cite{Pang2025} propose several general methods for constructing asymmetric NOAs with flexible run sizes, number of levels, and strengths. These methods can generate numerous new classes of NOAs, achieving the maximal number of factors with the minimal run size for both the smaller and larger arrays. For more other results on NOAs with special properties, we refer the readers to \cite{Tsai2017,Zheng2024}.

Although extensive research on NOAs has been conducted, the available bounds remain far from satisfactory, especially those derived by Lin et al. \cite{Lin} for the asymmetric NOAs. In this paper, we develop a character-theoretic framework for nested orthogonal arrays that not only elucidates their fundamental structure but also sharpens existing bounds. This framework allows us to provide a new proof of the Rao's bound for orthogonal arrays in \cite{Hedayat1999} using group character theory. Another contribution of this paper is the improvement of the bounds established in \cite{Lin} for asymmetric nested orthogonal arrays, and when the array becomes symmetric, our bounds coincide with those from \cite{MQW}.

The rest of the paper is organized as follows. Section 2 establishes the group-theoretic framework for orthogonal arrays and introduces the character theory of finite abelian groups. Section 3 provides a new proof of Rao's bound using character theory. Section 4 derives improved lower bounds for nested orthogonal arrays. Section 5 concludes the paper with brief remarks.

\section{Preliminaries}

In this section, we develop a group-theoretic framework for orthogonal arrays.

Let $\mathbb{Z}_+$ denote the set of all positive integers. The notation $|\cdot|$ denotes the cardinality of a set (or multiset) and the modulus of a complex number, the context will disambiguate the usage. Specifically, for a complex number $a = x + iy$ with $x, y \in \mathbb{R}$, its modulus is defined as $|a| = \sqrt{x^2 + y^2}$.
Throughout the paper, we fix the abelian group $G = \mathbb{Z}_{s_1}\times\mathbb{Z}_{s_2}\times\cdots\times\mathbb{Z}_{s_k}$, where each $\mathbb{Z}_{s_i}=\{0,1,\dots,s_i-1\}$ is the additive group modulo $s_i$ for $1\leq i\leq k$. When \(s_1 = s_2 = \cdots = s_k = s\), we write \(G = \mathbb{Z}_s^k\). For any index set $I=\{i_1,i_2,\dots,i_t\}\subseteq\{1,2,\dots,k\}$, let $G_I=\mathbb{Z}_{s_{i_1}}\times\mathbb{Z}_{s_{i_2}}\times\cdots\times\mathbb{Z}_{s_{i_t}}$, and define
\[
\pi_I : G \longrightarrow G_I, \ (g_1,g_2,\ldots,g_k)\longmapsto(g_{i_1},g_{i_2},\ldots,g_{i_t}).
\]

For a multiset $A$, let $set(A)$ denote the set of its distinct elements of $A$. If $A$ is an array of size $|A|\times k$, we always assume that $set(A)$ consists of its rows and $set(A)\subseteq G$, that is, each row of $A$ is one element of $G$. The projection of $A$ onto the coordinates indexed by $I$, denoted by $\pi_I(A)$, is the multiset where each element is taken from $G_I$ obtained by the mapping
\[
\pi_I|_A : A \longrightarrow G_I, \ (a_1,a_2,\ldots,a_k)\longmapsto(a_{i_1},a_{i_2},\ldots,a_{i_t}).
\]
For any $\beta \in G_I$, the multiplicity of $\beta$ in $\pi_I(A)$ is denoted by $$\lambda_I(\beta) = |\{ \alpha \in A : \pi_I(\alpha) = \beta \}|.$$
The following example illustrates the relationship between a multiset $A$
and its projections.

\begin{exam}
Let $G = \mathbb{Z}_2 \times \mathbb{Z}_3 \times \mathbb{Z}_4$ and consider the multiset $A = \{000, 000, 110, 110, 102\}$. Then $set(A)=\{000, 110, 102\}\subseteq G$.

For $I = \{1, 2\}$, the coordinate subgroup is $G_I = \mathbb{Z}_2 \times \mathbb{Z}_3$, and the projection multiset is $\pi_I(A) = \{00, 00,11,11, 10\}$. The corresponding multiplicities are $\lambda_I(00) = 2$, $\lambda_I(11) = 2$ and $\lambda_I(10) = 1$.
\end{exam}

\begin{defi}
Let $A$ be a finite multiset satisfying \(set(A)\subseteq G\) and \(t\in\mathbb{Z}_+\) with \(1\le t\le k\).
We call \(A\) an \emph{orthogonal array of strength \(t\)}, denoted by
OA\((|A|,s_1s_2\cdots s_k,t)\), provided that for each index subset
\(I=\{i_1,i_2,\dots,i_t\}\subseteq\{1,2,\dots,k\}\), the restricted projection
\(\pi_I|_A:A\longrightarrow G_I\) is \emph{uniform}. Equivalently, there exists an integer
\(\lambda_I\in\mathbb{Z}_{+}\) determined uniquely by \(I\) satisfying
\[
|\{\,\alpha\in A:\ \pi_I(\alpha)=\beta\,\}|=\lambda_I
\quad\text{for all } \beta\in G_I.
\]
\end{defi}

A direct verification gives
\[
\lambda_I=\frac{|A|}{|G_I|}=\frac{|A|}{\prod_{j\in I}s_j}.
\]
This immediately implies that \(\prod_{j\in I}s_j\) divides \(|A|\) for every index subset \(I\). In an OA$(|A|, s_1s_2 \\\cdots s_k, t)$, the rows are referred to as
\emph{runs} and $|A|$ is the \emph{size}. Each $\mathbb{Z}_{s_i} (1\leq i\leq k) $ is a
\emph{factor}, so $k$ is the total number of factors, and $s_i$ is the \emph{level} of the $i$-th factor. The parameter $t$ is the \emph{strength} of the OA.
When $s_1=\cdots=s_k=s$, the array is called \emph{symmetric} and denoted by
OA$(|A|,s,k,t)$; otherwise, it is \emph{asymmetric} or \emph{mixed}.

To characterize orthogonal arrays via group characters, we now introduce the character theory of the finite abelian groups. Let $G$ be an abelian group of order $n$. A \textit{character} $\chi$ of $G$ is a group homomorphism from $G$ to the multiplicative group of nonzero complex numbers $\mathbb{C}^*$, that is, a mapping $\chi: G \longrightarrow \mathbb{C}^*$ satisfying
\begin{center}
$\chi(g_1 g_2) = \chi(g_1) \chi(g_2)$
\end{center}
for all $g_1, g_2 \in G$. Similarly, we may define the restriction of a character $\chi|_I: G_I \longrightarrow \mathbb{C}^*$.

The following lemma is established in the monograph \emph{Fourier Analysis on Finite Groups and Applications} \cite{Terras1999}.

\begin{lem}[\!\cite{Terras1999}]\label{char}
Let $G = \mathbb{Z}_{s_1} \times \mathbb{Z}_{s_2} \times \cdots \times \mathbb{Z}_{s_k}$ be a finite abelian group. For any fixed element $\mathbf{m} = (m_1,m_2, \dots, m_k) \in G$, the mapping $\chi_{\mathbf{m}} : G \to \mathbb{C}^*$ defined by
\begin{center}$\chi_{\mathbf{m}}(\mathbf{x}) = \prod_{j=1}^k e^{ 2\pi i m_j x_j/s_j }, \ \ \mathbf{x} = (x_1,x_2,\dots,x_k) \in G,$\end{center}
is a character of $G$.
\end{lem}

%

Let $\widehat{G}$ denote the set of all characters of $G$. For any two characters $\chi_{\mathbf{m}}, \chi_{\mathbf{n}} \in \widehat{G}$, define \emph{their product} $\chi_{\mathbf{m}} \chi_{\mathbf{n}} : G \longrightarrow \mathbb{C}^*$ by
\begin{center}
$\chi_{\mathbf{m}} \chi_{\mathbf{n}}(\mathbf{x}) = \chi_{\mathbf{m}}(\mathbf{x}) \chi_{\mathbf{n}}(\mathbf{x})$,\ \ $\mathbf{x} \in G$.
\end{center}
It is straightforward to verify that $\widehat{G}$ forms an abelian group under this operation. Moreover, we have $\chi_{\mathbf{m}} \chi_{\mathbf{n}} = \chi_{\mathbf{m}+\mathbf{n}}$. The identity element of $\widehat{G}$ is the trivial character $\chi_{\mathbf{0}}$, and the inverse of $\chi_{\mathbf{m}}$ is its complex conjugate, which satisfies $\overline{\chi_{\mathbf{m}}} =(\chi_{\mathbf{m}})^{-1}= \chi_{-\mathbf{m}}$. The group $\widehat{G}$ is called \emph{the character group} of $G$, or \emph{the dual group} of $G$. We have the following result.

\begin{lem}[\!\cite{Terras1999}]\label{char2}
$G$ is isomorphic to $\widehat{G}.$
\end{lem}

For two complex-valued functions $f$, $g: G \longrightarrow \mathbb{C}$, \emph{their inner product} is defined as
\begin{center}
$\langle f, g \rangle = \frac{1}{|G|} \sum\limits_{\mathbf{x} \in G} f(\mathbf{x}) \overline{g(\mathbf{x})}.$
\end{center}

\begin{defi}[Fourier Transform \cite{Terras1999}]\label{def:DFT}
Let $G = \mathbb{Z}_{s_1} \times \mathbb{Z}_{s_2} \times\cdots \times \mathbb{Z}_{s_k}$ be a finite abelian group. For each $\mathbf{m}=(m_1,m_2,\ldots,m_k) \in G$, $\chi_{\mathbf{m}}\in \widehat{G}$. Let $f$ be a complex-valued function. \emph{The Fourier transform} of $f$ is defined by
$$\mathcal{F}f(\chi_\mathbf{m})=\widehat{f}(\chi_\mathbf{m})=\langle f, \chi_\mathbf{m} \rangle= \frac{1}{|G|} \sum_{\mathbf{x} \in G} f(\mathbf{x}) \overline{\chi_{\mathbf{m}}(\mathbf{x})}.$$
The corresponding \emph{Fourier inversion formula} is given by
$$
f(\mathbf{x}) =\sum_{\chi_\mathbf{m} \in \widehat{G}} \widehat{f}(\chi_\mathbf{m}) \chi_{\mathbf{m}}(\mathbf{x})= \sum_{\mathbf{m} \in G} \widehat{f}(\chi_\mathbf{m}) \chi_{\mathbf{m}}(\mathbf{x}).$$
\end{defi}
\begin{lem}[\!\cite{Terras1999}]\label{OR}
Let $G$ be a finite abelian group, and  $\chi, \psi \in \widehat{G}$. Then
\[
\langle \chi, \psi \rangle = \frac{1}{|G|}\sum_{g \in G} \chi(g) \overline{\psi(g)} =
\begin{cases}
1, & \text{if } \chi = \psi, \\
0, & \text{otherwise},
\end{cases} \ \
and \ \
\sum_{g \in G} \chi(g) = \begin{cases}
|G|, & \text{if } \chi \ \text{is trivial}, \\
0, & \text{otherwise}.
\end{cases}
\]
\end{lem}

For $\mathbf m=(m_1,m_2,\dots,m_k)\in G$,
\emph{the support} of $\mathbf m$ is defined as
\begin{center}
\(\supp(\mathbf m)=\{i: m_i\neq 0, 1\leq i \leq k\}.\)
\end{center}
 For example,
let $G = \mathbb{Z}_2 \times \mathbb{Z}_2 \times \mathbb{Z}_3$ and
$\mathbf{m} = (1, 0, 2)$, then $\supp(\mathbf{m}) = \{1, 3\}$. The element
$\mathbf{m} = (0,0,0)$ has empty support, i.e., $\supp(\mathbf{m}) =\emptyset$.

We now present a characterization of orthogonal arrays in terms of group characters.
\begin{lem}\label{oa}
Let \(G = \mathbb{Z}_{s_1}\times \mathbb{Z}_{s_2}\times\cdots \times \mathbb{Z}_{s_k}\).
Then \(A\) is an OA$(|A|,s_1s_2\dots s_k,t)$ over $G$ if and only if
\[\sum_{\mathbf{x}\in A}\chi_{\mathbf{m}}(\mathbf{x})=0\]
for all nonzero elements $\mathbf{m}\in G$ with $|\supp(\mathbf{m})|\le t$, where $\mathbf{x}\in A$ means that $\mathbf{x}$ is a row of $A$.
\end{lem}
\begin{proof}
Suppose $A$ is an OA$(|A|,s_1s_2\dots s_k,t)$ over $G$. Let $\mathbf{m} = (m_1,m_2, \dots, m_k) \in G\setminus \{\mathbf{0}\}$, set $\text{supp}(\mathbf{m})=I \subseteq \{1,2,\dots,k\}$ and denote the coordinate projection $\pi_I: G \longrightarrow G_I$. Since $m_j = 0$ for every index $j \notin I$, we have
\begin{center}
$\chi_{\mathbf{m}}(\mathbf{x}) = \chi_{\mathbf{m'}}(\pi_I(\mathbf{x}))$
\end{center}
where $\mathbf{m'}=\pi_I(\mathbf{m})$.
We now partition the rows of $A$ according to their images under $\pi_I$. This yields
$$\sum_{\mathbf{x} \in A} \chi_{\mathbf{m}}(\mathbf{x}) = \sum_{\mathbf{y} \in G_I} \sum_{\mathbf{x} \in A_{\mathbf{y}}} \chi_{\mathbf{m'}}(\mathbf{y}) = \sum_{\mathbf{y} \in G_I} |A_{\mathbf{y}}| \chi_{\mathbf{m'}}(\mathbf{y}),$$ where $A_{\mathbf{y}} = \{ \mathbf{x} \in A : \pi_I(\mathbf{x}) = \mathbf{y} \}$.
By the definition of an orthogonal array of strength $t$, the projection $\pi_I$ is uniform when $|I| \le t$. Hence we only consider nonzero $\mathbf{m}$ satisfying $1 \le |\text{supp}(\mathbf{m})| \le t$. In this case, the uniformity condition gives $|A_{\mathbf{y}}| = \frac{|A|}{|G_I|}$ for all $\mathbf{y} \in G_I$.
Therefore, we obtain
$$\sum_{\mathbf{x} \in A} \chi_{\mathbf{m}}(\mathbf{x}) = \frac{|A|}{|G_I|} \sum_{\mathbf{y} \in G_I} \chi_{\mathbf{m'}}(\mathbf{y}).$$
Since $\mathbf{m'}=\pi_I(\mathbf{m})\neq \mathbf{0}$ followed from $\mathbf{m} \neq \mathbf{0}$, the character $\chi_{\mathbf{m'}}$ is non-trivial. By Lemma \ref{OR}, it follows that
$$\sum_{\mathbf{y} \in G_I} \chi_{\mathbf{m'}}(\mathbf{y}) = 0.$$
Consequently, the character sum over $A$ vanishes.

Conversely, assume that for every nonzero vector $\mathbf{m}$ with $0 < |\text{supp}(\mathbf{m})| \le t$, the character sum satisfies
$$\sum_{\mathbf{x} \in A} \chi_{\mathbf{m}}(\mathbf{x}) = 0.$$
It is enough to show that for any index set $I\subseteq\{1,2,\dots,k\}$ with$|I| = t$, the restricted projection \(\pi_I|_A:A\longrightarrow G_I\) is uniform; that is, there exists a constant
\(\lambda_I\in\mathbb{Z}_{+}\) (depending only on \(I\)) such that
\[|\{\,\mathbf{x}\in A:\ \pi_I(\mathbf{x})=\mathbf{y}\,\}|=\lambda_I
\quad\text{for every } \mathbf{y}\in G_I.\]
For an arbitrary $\mathbf{y} \in G_I$, we
define the function $f_{\mathbf{y}}: G_I \longrightarrow \{0,1\}$ by
\[f_{\mathbf{y}}(\mathbf{x}) = \begin{cases}
1, & \text{if } \pi_I(\mathbf{x}) = \mathbf{y}, \\
0, & \text{if } \pi_I(\mathbf{x}) \neq \mathbf{y}.
\end{cases}\]
It is straightforward to verify that
\begin{equation*}
|\{\,\mathbf{x}\in A:\ \pi_I(\mathbf{x})=\mathbf{y}\,\}| = \sum_{\mathbf{x} \in A} f_{\mathbf{y}}(\mathbf{x}).
\end{equation*}
By Definition \ref{def:DFT}, the Fourier transform of $f$ is given by

$$\widehat{f}_{\mathbf{y}}(\chi_\mathbf{m'}) = \langle f_{\mathbf{y}}, \chi_{\mathbf{m'}} \rangle = \frac{1}{|G_I|} \sum_{\pi_I(\mathbf{x}) \in G_I} f_{\mathbf{y}}(\pi_I(\mathbf{x})) \overline{\chi_{\mathbf{m'}}(\pi_I(\mathbf{x}))}=\frac{1}{|G_I|}\overline{\chi_{\mathbf{m'}}(\mathbf{y})}.
$$
The Fourier inversion formula then yields
\begin{equation*}
f_{\mathbf{y}}(\mathbf{x}) =\sum_{\mathbf{m}' \in G_I} \widehat{f}_{\mathbf{y}}(\chi_\mathbf{m'}) \chi_{\mathbf{m}'}(\pi_I(\mathbf{x}))= \frac{1}{|G_I|} \sum_{\mathbf{m}' \in G_I} \overline{\chi_{\mathbf{m}'}(\mathbf{y})} \chi_{\mathbf{m}'}(\pi_I(\mathbf{x})).
\end{equation*}
For each $\mathbf{m}' \in G_I$, we define a vector $\mathbf{m} = (m_1, m_2, \ldots, m_k) \in G$ whose components are given by
\[m_j = \begin{cases}
m'_j, & \text{if } j \in I, \\
0, & \text{if } j \notin I,
\end{cases}\]
where $m'_j$ denotes the $j$-th component of $\mathbf{m}'$.
From the relation between $\mathbf{m}'$ and $\mathbf{m}$, it follows that
$$0< |\text{supp}(\mathbf{m})| = |\text{supp}(\mathbf{m}')| \le |I| = t, $$
where $\mathbf{m}'\neq \mathbf{0}$, and for any $\mathbf{x}=(a_1, a_2, \dots, a_k) \in A$, we have
$$\chi_{\mathbf{m}}(\mathbf{x}) = \prod_{j=1}^k \chi_{m_j}(a_j)=\prod_{j\in I} \chi_{m_j'}(a_j)\prod_{j'\notin I} \chi_{0}(a_j')= \chi_{\mathbf{m}'}(\pi_I(\mathbf{x})).$$
Substituting the character expansion of $f_{\mathbf{y}}(\mathbf{x})$ into the frequency counting formula and interchanging the order of finite summations, the row count $|\{\,\mathbf{x}\in A: \pi_I(\mathbf{x})=\mathbf{y}\,\}|$ is given by
\begin{align*}
|\{\,\mathbf{x}\in A:\ \pi_I(\mathbf{x})=\mathbf{y}\,\}|&= \sum_{\mathbf{x} \in A} f_{\mathbf{y}}(\mathbf{x})\\
 &= \sum_{\mathbf{x} \in A} \left[ \frac{1}{|G_I|} \sum_{\mathbf{m}' \in G_I} \overline{\chi_{\mathbf{m}'}(\mathbf{y})} \chi_{\mathbf{m}'}(\pi_I(\mathbf{x})) \right] \\
&= \frac{1}{|G_I|} \sum_{\mathbf{m}' \in G_I} \overline{\chi_{\mathbf{m}'}(\mathbf{y})} \sum_{\mathbf{x} \in A} \chi_{\mathbf{m'}}(\pi_I(\mathbf{x})) \\
&= \frac{1}{|G_I|} \sum_{\mathbf{m}' \in G_I} \overline{\chi_{\mathbf{m}'}(\mathbf{y})} \sum_{\mathbf{x} \in A} \chi_{\mathbf{m}}(\mathbf{x}) \\
&= \frac{1}{|G_I|} \left[ \overline{\chi_{\mathbf{0}}(\mathbf{y})} \sum_{\mathbf{x} \in A} \chi_{\mathbf{0}}(\mathbf{x}) + \sum_{\mathbf{m}' \in G_I \setminus \{\mathbf{0}\}} \overline{\chi_{\mathbf{m}'}(\mathbf{y})} \sum_{\mathbf{x} \in A} \chi_{\mathbf{m}}(\mathbf{x}) \right]\\
&= \frac{1}{|G_I|} \left[ 1 \cdot |A| + 0 \right]\\
&= \frac{|A|}{|G_I|}=\lambda_I.
\end{align*}
Therefore, $A$ is an orthogonal array of strength $t$ over $G$, which completes the proof.
\end{proof}

\begin{rmk}
The character-based criterion established in this paper for the direct product group $G$ is conceptually inspired by Delsarte's foundational work \cite{Delsarte}. Although Theorem 3.30 in \cite{Hedayat1999} stated a characterization of symmetric orthogonal arrays without proof, we provide a necessary and sufficient condition for asymmetric orthogonal arrays via character sums and the Fourier transform on finite groups.
\end{rmk}

Setting \(s_1 = s_2 = \cdots = s_k = s\) in Lemma~\ref{oa} gives the following result.

\begin{cor}[\!\cite{Hedayat1999}]
An array \(A\) is an OA$(|A|,s,k,t)$ over $\mathbb{Z}_s$ if and only if
\[\sum_{\mathbf{x} \in A} e^{2\pi i (\mathbf{u} \cdot \mathbf{x})/s} = 0\]
for every nonzero \(\mathbf{u} \in \mathbb{Z}_s^k\) with \(|\operatorname{supp}(\mathbf{u})| \le t\), where the dot product \(\mathbf{u} \cdot \mathbf{x}\) is taken modulo \(s\).
\end{cor}

\section{A new proof for Rao's bound of orthogonal arrays}

In this section, we present a novel character-theoretic proof of the Rao bound \cite{Hedayat1999}. Let $\mathbb{C}^N$ denote the $N$-dimensional complex vector space.

\begin{thm}[Rao bound for OAs]\label{Rao bound}
Let \(G=\mathbb{Z}_{s_1}\times\mathbb{Z}_{s_2}\times\cdots\times\mathbb{Z}_{s_k}\), $s_i\in \mathbb{Z}_+$, $1\le i\le k$ and $s_1=\max\{s_i:1\le i\le k\}$. If $A$ is
an OA$(|A|,s_1s_2\cdots s_k,t)$ over $G$, then
\[|A| \ge
\begin{cases}
\displaystyle \sum_{i=0}^{\frac{t}{2}}\sum_{|I|=i}\prod_{j\in I}(s_j-1), & \text{if } $t$ \ \text{is even} , \\[12pt]
\displaystyle \sum_{i=0}^{\lfloor\frac{t}{2}\rfloor}\sum_{|I|=i}\prod_{j\in I}(s_j-1)
\;+\;
(s_1-1)\sum_{\substack{|I|=\lfloor\frac{t}{2}\rfloor \\ 1\not\in I}}
\prod_{j\in I}(s_j-1), & \text{if } $t$  \ \text{is odd},
\end{cases}\]
where $\lfloor \frac{t}{2}\rfloor$ is the greatest integer less than or equal to $\frac{t}{2}$.
\end{thm}
\begin{proof}
We consider the restrictions of characters $\chi_{\mathbf{m}}$, $\mathbf{m}\in G$, to the OA $A$ as vectors in the complex vector space $\mathbb{C}^{|A|}$. By the character sum property of orthogonal arrays of strength $t$, any two distinct characters $\chi_{\mathbf{m}}, \chi_{\mathbf{n}}$ satisfy $0 < |\operatorname{supp}(\mathbf{m} - \mathbf{n})| \le t$, which guarantees their mutual orthogonality and hence linear independence. Since the number of such linearly independent vectors cannot exceed the dimension $|A|$, the desired bound follows by counting the maximal cardinality of this collection.

Write
\[A = \left(
\begin{array}{*{1}{c}}
\mathbf{x}_1\\
\mathbf{x}_2 \\
\vdots\\
\mathbf{x}_{|A|}
\end{array}
\right), \quad \mathbf{x}_i = (x_{i1}, x_{i2},\dots, x_{ik}) \in G, \quad i=1,2,\dots,|A|.\]
For each $\mathbf{m} = (m_1, m_2, \dots, m_k) \in G$, let
\[\chi_{\mathbf{m}}(A)= \bigl(\chi_{\mathbf{m}}(\mathbf{x}_1), \chi_{\mathbf{m}}(\mathbf{x}_2), \dots, \chi_{\mathbf{m}}(\mathbf{x}_{|A|})\bigr) \in \mathbb{C}^{|A|},\]
where
\[\chi_{\mathbf{m}}(\mathbf{x}_i) = e^{2 \pi i \sum\limits_{j=1}^k m_j x_{ij}/s_j}, \quad i=1,2,\dots,|A|.\]
For any  $\mathbf{m},\mathbf{n} \in G$, the inner product of $\chi_{\mathbf{m}}(A)$ and $\chi_{\mathbf{n}}(A)$ is
\[
\langle \chi_{\mathbf{m}}(A), \chi_{\mathbf{n}}(A) \rangle
= \frac{1}{|A|}\sum_{i=1}^{|A|} \chi_{\mathbf{m}}(\mathbf{x}_i) \, \overline{\chi_{\mathbf{n}}(\mathbf{x}_i)}
= \frac{1}{|A|}\sum_{i=1}^{|A|} \chi_{\mathbf{m}-\mathbf{n}}(\mathbf{x}_i)
= \frac{1}{|A|}\sum_{\mathbf{x} \in A} \chi_{\mathbf{m}-\mathbf{n}}(\mathbf{x}).
\]
Since $A$ is an OA$(|A|,s_1\cdots s_k,t)$ over $G$, we have

\begin{center}
$\sum\limits_{\mathbf{x} \in A} \chi_{\mathbf{g}}(x) = 0$
\end{center}
by Lemma~\ref{oa},
for every $\mathbf{g} \in G$ with $0 < |\operatorname{supp}(\mathbf{g})| \le t$. Consequently, for any two distinct $\mathbf{m},\mathbf{n}\in G$ satisfying $0 < |\operatorname{supp}(\mathbf{m}-\mathbf{n})| \le t$, the restricted vectors $\chi_{\mathbf{m}}(A)$ and $\chi_{\mathbf{n}}(A)$ are mutually orthogonal and hence linearly independent in $\mathbb{C}^{|A|}$.

\medskip

\noindent\textbf{Case 1: $t$ is even.} Let $\mathcal{S} = \{ \chi_{\mathbf{m}} \in \widehat{G} : |\operatorname{supp}(\mathbf{m})| \le \frac{t}{2} \}$. For any distinct characters $\chi_{\mathbf{m}}, \chi_{\mathbf{n}} \in \mathcal{S}$, the triangle inequality yields
\[0 < |\operatorname{supp}(\mathbf{m} - \mathbf{n})| \le |\operatorname{supp}(\mathbf{m})| + |\operatorname{supp}(\mathbf{n})| \le t.\]
It follows that the vectors in $\{\chi_{\mathbf{m}}(A) : \chi_{\mathbf{m}} \in \mathcal{S}\}$ are pairwise orthogonal.

To determine $|\mathcal{S}|$, observe that for any fixed index set $I \subseteq \{1,2,\dots,k\}$ with $|I| = i$, there are precisely $\prod\limits_{j \in I} (s_j - 1)$ characters whose support is exactly $I$. Summing over all index sets of size at most
$\frac{t}{2}$, we obtain
\[|\mathcal{S}| = \sum_{i=0}^{\frac{t}{2}} \sum_{|I|=i} \prod_{j \in I} (s_j - 1).\]
Since these $|\mathcal{S}|$ vectors are linearly independent in $\mathbb{C}^{|A|}$, it follows immediately that $|A| \ge |\mathcal{S}|$.

\noindent\textbf{Case 2: $t$ is odd.} Let $\mathcal{T} = \mathcal{S'} \cup \mathcal{U}$, where
$$\mathcal{S'} = \left\{ \chi_{\mathbf{m}} \in \widehat{G} : |\operatorname{supp}(\mathbf{m})| \le \left\lfloor \frac{t}{2}\right\rfloor \right\}, \
\mathcal{U} = \left\{ \chi_{\mathbf{m}} : |\operatorname{supp}(\mathbf{m})| = \left\lceil \frac{t}{2}\right\rceil \text{ and } 1 \in \operatorname{supp}(\mathbf{m}) \right\}.$$
For any two distinct characters $\chi_{\mathbf{m}}, \chi_{\mathbf{n}} \in \mathcal{T}$, we show that $0 < |\operatorname{supp}(\mathbf{m} - \mathbf{n})| \le t$ via the following three cases:

$(1)$ When $\chi_{\mathbf{m}}, \chi_{\mathbf{n}} \in \mathcal{S'}$, we have $|\operatorname{supp}(\mathbf{m} - \mathbf{n})| \le t$.

$(2)$ When $\chi_{\mathbf{m}} \in \mathcal{U}$ and $\chi_{\mathbf{n}} \in \mathcal{S'}$, the same bound $|\operatorname{supp}(\mathbf{m} - \mathbf{n})| \le t$ holds.

$(3)$ For $\chi_{\mathbf{m}}, \chi_{\mathbf{n}} \in \mathcal{U}$, the first coordinate lies in the support of both indices, so $1 \in \operatorname{supp}(\mathbf{m}) \cap \operatorname{supp}(\mathbf{n})$. Consequently,
$$|\operatorname{supp}(\mathbf{m} - \mathbf{n})| \le |\operatorname{supp}(\mathbf{m})| + |\operatorname{supp}(\mathbf{n})| - 1 =t.$$
From the hypotheses, we obtain $s_1 = \max\{s_i : 1 \le i \le k\}$. Thus, the maximum number of linearly independent restricted characters is
\[|\mathcal{T}| = \sum_{i=0}^{\left\lfloor \frac{t}{2}\right\rfloor} \sum_{|I|=i} \prod_{j \in I} (s_j - 1) + (s_1 - 1)\sum_{\substack{|I|=\left\lfloor \frac{t}{2}\right\rfloor \\ 1 \not\in I}} \prod_{j \in I} (s_j - 1),\]
which immediately implies the desired bound $|A| \ge |\mathcal{T}|$. The proof is completed.
\end{proof}

To illustrate the proof of Theorem \ref{Rao bound}, we present two examples according to
the parity of the strength of orthogonal array.

\begin{exam}
Let
\[
A = \begin{pmatrix}
0 & 0 & 0 \\
0 & 1 & 1 \\
1 & 0 & 1 \\
1 & 1 & 0
\end{pmatrix}.
\]
It is straightforward to see that $A$ is an OA$(4,3,2,2)$ over $\mathbb{Z}_2$. Following \textbf{Case 1} of the proof of Theorem \ref{Rao bound}, we take
\begin{center}
$\mathcal{S} = \{ \chi_{\mathbf{m}} \in \widehat{\mathbb{Z}_2^3} : |\operatorname{supp}(\mathbf{m})| \le 1 \}=\{ \chi_{(0,0,0)}, \chi_{(1,0,0)}, \chi_{(0,1,0)}, \chi_{(0,0,1)} \}.$
\end{center}
For each $\chi_\mathbf{m} \in \mathcal{S}$, consider the vector $\big( \chi_{\mathbf{m}}(\mathbf{x}) \big)_{\mathbf{x} \in A}$, where $\chi_{\mathbf{m}}(\mathbf{x}) = (-1)^{\mathbf{m} \cdot \mathbf{x}}$. These four vectors are the row vectors, as shown below.
\begin{center}
\begin{tabular}{c|cccc}
\hline
\diagbox[width=3em,height=2em]{\(\mathbf{m}\)}{\(\mathbf{x}\)} & $(0,0,0)$ & $(0,1,1)$ & $(1,0,1)$ & $(1,1,0)$ \\
\hline
$(0,0,0)$ & $1$ & $1$ & $1$ & $1$ \\
$(1,0,0)$ & $1$ & $1$ & $-1$ & $-1$ \\
$(0,1,0)$ & $1$ & $-1$ & $1$ & $-1$ \\
$(0,0,1)$ & $1$ & $-1$ & $-1$ & $1$ \\
\hline
\end{tabular}
\end{center}
For any two distinct $\chi_{\mathbf{m}}, \chi_{\mathbf{n}} \in \mathcal{S}$, the vectors
$\chi_{\mathbf{m}}(A)$ and $\chi_{\mathbf{n}}(A)$ are mutually orthogonal.
Consequently, the four vectors are linearly independent in $\mathbb{C}^4$, thus establishing the lower bound
\[|A| \ge |\mathcal{S}| = \sum_{i=0}^{1} \binom{3}{i}(2-1)^i = \binom{3}{0} + \binom{3}{1} = 1 + 3 = 4.\]
\end{exam}

\begin{exam}
Let
\[A = \left(
\begin{array}{*{16}{c}}
0 & 0 & 0 & 0 & 1 & 1 & 1 & 1 & 2 & 2 & 2 & 2 & 3 & 3 & 3 & 3 \\
0 & 0 & 1 & 1 & 0 & 0 & 1 & 1 & 0 & 0 & 1 & 1 & 0 & 0 & 1 & 1 \\
0 & 1 & 0 & 1 & 0 & 1 & 0 & 1 & 0 & 1 & 0 & 1 & 0 & 1 & 0 & 1 \\
0 & 1 & 1 & 0 & 0 & 1 & 1 & 0 & 1 & 0 & 0 & 1 & 1 & 0 & 0 & 1
\end{array}
\right).\]
It is easy to verify that $A^T$ is an OA$(16, 4^1 2^3, 3)$ over $G = \mathbb{Z}_4 \times \mathbb{Z}_2^3$, where $A^T$ is the transpose of $A$. Following \textbf{Case 2} of the proof of Theorem \ref{Rao bound}, let $\mathcal{T} = \mathcal{S'} \cup \mathcal{U}$, where
$$\begin{aligned}
\mathcal{S'} &= \{ \chi_{\mathbf{m}}\in \widehat{G} : |\operatorname{supp}(\mathbf{m})| \le 1 \} \\ &= \{ \chi_{(0,0,0,0)},\chi_{(1,0,0,0)}, \chi_{(2,0,0,0)}, \chi_{(3,0,0,0)},\chi_{(0,1,0,0)}. \chi_{(0,0,1,0)}, \chi_{(0,0,0,1)} \}
\end{aligned}$$
and
$$\begin{aligned}
\mathcal{U}& = \{ \chi_{\mathbf{m}}\in \widehat{\mathbb{Z}_4 \times \mathbb{Z}_2^3} : |\operatorname{supp}(\mathbf{m})| = 2,\ 1 \in \operatorname{supp}(\mathbf{m}) \}\\
&=\{ \chi_{(1,1,0,0)},\chi_{(1,0,1,0)}, \chi_{(1,0,0,1)},\chi_{(2,1,0,0)},\chi_{(2,0,1,0)}, \chi_{(2,0,0,1)}, \chi_{(3,1,0,0)},\chi_{(3,0,1,0)}, \chi_{(3,0,0,1)} \}.
\end{aligned}$$
For any two distinct $\chi_{\mathbf{m}}, \chi_{\mathbf{n}} \in \mathcal{T}$, we have the vectors in $\{\chi_{\mathbf{m}}(A^T) : \chi_{\mathbf{m}} \in \mathcal{T}\}$ are pairwise orthogonal. Consequently,
$$|A^T| \ge|\mathcal{T}| =|\mathcal{S'}|+ |\mathcal{U}|= (1 + (4-1) + 3(2-1))+((4-1) \times 3(2-1))=7 + 9 = 16.$$
\end{exam}

\section{Improved Bounds for nested orthogonal arrays }

In this section, we improve the lower bounds for asymmetric nested orthogonal arrays by exploiting group characters and the properties of Gram matrices.

\begin{defi}
Suppose $A$ is an OA$(|A|,s_1s_2\cdots s_k,t)$ over $G = \mathbb{Z}_{s_1} \times \mathbb{Z}_{s_2} \times\cdots \times \mathbb{Z}_{s_k}$ that contains a subarray $B$ such that $B$ is  an OA$(|B|,r_1r_2\cdots r_k,t)$ over $H = \mathbb{Z}_{r_1} \times \mathbb{Z}_{r_2} \times\cdots \times \mathbb{Z}_{r_k}$,
where the elements of $B$ satisfy $set(B) \subseteq H$, and $r_i \le s_i$ for all $1 \le i \le k$, with strict inequality $r_i < s_i$ for at least one $i$.
Then \( A \) is called an \emph{asymmetric nested orthogonal array} and is denoted by NOA\(((|A|, |B|), k, (s_1s_2 \cdots s_k, r_1 r_2\cdots r_k), t) \). In the special case where
\( s_1 =s_2 = \cdots = s_k = s \) and \( r_1 = r_2 =\cdots = r_k = r \), this reduces to a \emph{symmetric nested orthogonal array}, denoted by NOA\(((|A|, |B|), k, (s, r), t) \).
\end{defi}

For a matrix $A\in \mathbb{C}^{m\times n}$, we define its \emph{Gram matrix} as $AA^H$,  where $A^H$ is the conjugate transpose of $A$. The Hermitian norm on $\mathbb{C}^n$ is denoted by $\|\cdot\|$ and given by $\|\mathbf{u}\| = \sqrt{ \mathbf{u}^H\mathbf{u}}$ for $\mathbf{u} \in \mathbb{C}^n$. We now present some results concerning the eigenvalues of the Gram matrix.

\begin{lem}[Rayleigh quotient\cite{Serre2010}]\label{lem:rayleigh}
Let $M$ be an $m \times m$ Gram matrix with maximum eigenvalue $\lambda_{max}(M)$. Then
\[\lambda_{max}(M) = \max\left\{\frac{\mathbf{v}^H M \mathbf{v}}{\|\mathbf{v}\|^2}:\|\mathbf{v}\|\neq0, \mathbf{v} \in \mathbb{C}^m\right\}= \max\left\{\frac{\mathbf{v}^H M \mathbf{v}}{\mathbf{v}^H\mathbf{v}}:\|\mathbf{v}\|\neq0, \mathbf{v} \in \mathbb{C}^m\right\}.\]
\end{lem}

The following result holds from a straightforward linear algebraic argument, we omit the proof for simplicity.

\begin{lem}\label{same nonzero eigenvalues}
For $A \in \mathbb{C}^{m \times n}$, the matrices $AA^H$ and $A^HA$ have the same nonzero eigenvalues.
\end{lem}
%

\begin{lem}[Parseval's Identity \cite{Terras1999}]\label{lem:parseval}
Let \( G = \mathbb{Z}_{s_1} \times \mathbb{Z}_{s_2} \times\cdots \times \mathbb{Z}_{s_k} \) be a finite abelian group, and let \( f : G \to \mathbb{C} \) be a complex-valued function. Then
\[\sum_{\chi_\mathbf{m} \in \widehat{G}} |\widehat{f}(\chi_\mathbf{m})|^2 = \frac{1}{|G|} \sum_{\mathbf{x} \in G} |f(\mathbf{x})|^2.\]
\end{lem}

\begin{lem}\label{lem:combined_univariate}
Let $s, r \in \mathbb{Z}_+$ with $1 < r \le s$, and $\omega_{s} = e^{2\pi i / s}$.
Then
\[\sum_{a=1}^{s-1} \left| \sum_{x=0}^{r-1} \omega_{s}^{a x} \right|^2 = r(s - r).\]
\end{lem}

\begin{proof}
Define the function $f: \mathbb{Z}_{s} \longrightarrow \{0,1\}$ by
\[f(x) =
\begin{cases}
1, & \text{if }  x \in \{0, 1, \dots, r-1\}, \\
0, & \text{otherwise}.
\end{cases}\]
It then follows that
\[\sum_{x \in \mathbb{Z}_{s}} |f(x)|^2 = \sum_{x=0}^{r-1} 1^2 + \sum_{x=r}^{s-1} 0^2 = \sum_{x=0}^{r-1} 1 = r. \]
By Definition \ref{def:DFT}, for every $a \in \mathbb{Z}_{s}$, the Fourier transform $\widehat{f}(\chi_{a})$ of $f$ associated with the character $\chi_{a}(x) = \omega_{s}^{a x}$ is given by
\[\widehat{f}(\chi_{a}) = \langle f, \chi_{a} \rangle = \frac{1}{s} \sum_{x \in \mathbb{Z}_{s}} f(x) \overline{\chi_{a}(x)} = \frac{1}{s} \sum_{x=0}^{r-1} \omega_{s}^{-a x}. \]
In particular, if $a=0$, $\widehat{f}(\chi_{0}) =\frac{r}{s}$.
This explicitly yields
\[\sum_{x=0}^{r-1} \omega_{s}^{-a x} = s \widehat{f}(\chi_{a}).\]
Then we obtain
\begin{center}
$\begin{aligned}
\sum_{a=1}^{s-1} \left| \sum_{x=0}^{r-1} \omega_{s}^{a x} \right|^2
&= \sum_{a=0}^{s-1} \left| \sum_{x=0}^{r-1} \omega_{s}^{a x} \right|^2- \left| \sum_{x=0}^{r-1} \omega_{s}^{0} \right|^2\\
&= \sum_{a=0}^{s-1} \left| \sum_{x=0}^{r-1} \omega_{s}^{-a x} \right|^2- \left| \sum_{x=0}^{r-1} \omega_{s}^{0} \right|^2\\
&=\sum_{a=0}^{s-1} |s \widehat{f}(\chi_{a})|^2-|s \widehat{f}(\chi_0)|^2\\
&=s^2\left( \sum_{a=0}^{s-1} |\widehat{f}(\chi_{a})|^2 - |\widehat{f}(\chi_0)|^2\right).
\end{aligned}$
\end{center}
By Lemma~\ref{lem:parseval}, which asserts the identity
$\sum\limits_{a \in \mathbb{Z}_{s}} |\widehat{f}(\chi_{a})|^2 = \frac{1}{s} \sum\limits_{x \in \mathbb{Z}_{s}} |f(x)|^2$, thus

\begin{center}
$\begin{aligned}
s^2\left( \sum_{a=0}^{s-1} |\widehat{f}(\chi_{a})|^2 - |\widehat{f}(\chi_0)|^2\right)
&=s^2\left( \frac{1}{s} \sum_{x \in \mathbb{Z}_{s}} |f(x)|^2 - \left|\frac{r}{s}\right|^2\right)\\
&= s^2 \cdot\left(\frac{1}{s}\cdot r-\frac{r^2}{s^2}\right)\\
&= r(s - r).
\end{aligned}$
\end{center}
It follows that $\sum\limits_{a=1}^{s-1} \left| \sum\limits_{x=0}^{r-1} \omega_{s}^{a x} \right|^2 = r(s - r).$ This completes the proof.
\end{proof}

Let
$$\mathcal{F}=
\begin{cases}
\mathcal{S}, & \text{if } t=2u, \\
\mathcal{T}, & \text{if } t=2u+1,
\end{cases}$$
where $\mathcal{S}$ and $\mathcal{T}$ are the set of characters defined in Theorem~\ref{Rao bound}. For any orthogonal array, $C$, define the $|C| \times |\mathcal{F}|$ matrix
\[\Psi_C = \bigl( \chi_{\mathbf{a}}(\mathbf{x}) \bigr)_{\mathbf{x} \in C, \chi_\mathbf{a} \in \mathcal{F}}.\]

\begin{thm}\label{lem:eigenvalues}
Let $A$ be an $OA(|A|, s_1s_2 \cdots s_k, t)$ over $G$ and $B$ an OA$(|B|, r_1r_2 \cdots r_k, t)$ over $H$, such that $A$ contains $B$. Let $\Psi_A\Psi_A^H$ denote the Gram matrix of $\Psi_A$. Then the following properties hold.

$(1)$ All nonzero eigenvalues of $\Psi_A\Psi_A^H$ are equal to $|A|$, i.e., the size of $A$.

$(2)$ Let $\frac{s_1}{r_1}=\max\{\frac{s_i}{r_i}:1\le i\le k\}$. Then
\small{ $$|A| \ge \begin{cases}
\displaystyle |B| \sum_{i=0}^{t/2} \sum_{\substack{I \subseteq \{1,\dots,k\}\\ |I|=i}} \prod_{j\in I} \left( \frac{s_j}{r_j} - 1 \right), & t \text{ even},\\[2.5em]
\displaystyle |B| \left[ \sum_{i=0}^{\lfloor t/2 \rfloor} \sum_{\substack{I \subseteq \{1,\dots,k\}\\ |I|=i}} \prod_{j\in I} \left( \frac{s_j}{r_j} - 1 \right) + \left( \frac{s_1}{r_1} - 1 \right) \sum_{\substack{I \subseteq \{2,\dots,k\}\\ |I|=\lfloor t/2 \rfloor}} \prod_{j\in I} \left( \frac{s_j}{r_j} - 1 \right) \right], & t \text{ odd}.
\end{cases}$$}
\end{thm}
\begin{proof}
$(1)$ By assumption, $\Psi_A = \bigl( \chi_{\mathbf{a}}(\mathbf{x}) \bigr)_{\mathbf{x} \in A, \ \chi_{\mathbf{a}} \in \mathcal{F}}$ is an $|A| \times |\mathcal{F}|$ matrix, where
\[
\mathcal{F} =
\begin{cases}
\{\chi_{\mathbf{a}} : |\operatorname{supp}(\mathbf{a})| \le t/2\}, & t \text{ even},\\[0.5em]
\{\chi_{\mathbf{a}} : |\operatorname{supp}(\mathbf{a})| \le \lfloor t/2 \rfloor\} \cup \{\chi_{\mathbf{a}} : |\operatorname{supp}(\mathbf{a})| = \lceil t/2 \rceil \text{ and } 1 \in \operatorname{supp}(\mathbf{a})\}, & t \text{ odd}.
\end{cases}
\]
Now we consider the matrix $\Psi_A^H \Psi_A \in \mathbb{C}^{|\mathcal{F}| \times |\mathcal{F}|}$, its entries are given by $$(\Psi_A^H \Psi_A)_{\chi_{\mathbf{a}}, \chi_{\mathbf{b}}} = \sum_{\mathbf{x} \in A} \overline{\chi_{\mathbf{a}}(\mathbf{x})} \chi_{\mathbf{b}}(\mathbf{x}) = \sum_{\mathbf{x} \in A} \chi_{\mathbf{b}-\mathbf{a}}(\mathbf{x}).$$
For any two distinct vectors $\chi_{\mathbf{a}}, \chi_{\mathbf{b}} \in \mathcal{F}$, the support of $\mathbf{b}-\mathbf{a}$ satisfies $0<|\operatorname{supp}(\mathbf{b}-\mathbf{a})| \le t$. Since $A$ is an orthogonal array of strength $t$, the character sum vanishes for $\mathbf{a} \neq \mathbf{b}$ and equals $|A|$ when $\mathbf{a} = \mathbf{b}$ by Lemmas \ref{OR} and \ref{oa}, respectively. Then we have
\[\sum_{\mathbf{x} \in A} \chi_{\mathbf{b}-\mathbf{a}}(\mathbf{x}) =
\begin{cases}
|A|, & \mathbf{a} = \mathbf{b}, \\
0, & \mathbf{a} \neq \mathbf{b}.
\end{cases}\]
Therefore, $\Psi_A^H \Psi_A = |A| I_{|\mathcal{F}|}$. By Lemma \ref{same nonzero eigenvalues}, $\Psi_A^H \Psi_A$ and  $\Psi_A \Psi_A^H$ possess the same non-zero eigenvalues, the non-zero eigenvalues of $\Psi_A \Psi_A^H$ are all equal to $|A|$.

$(2)$ Let $\Psi_B$ be the $|B| \times |\mathcal{F}|$ matrix formed by selecting the rows of $\Psi_A$ indexed by $B$. By definition, for any $\mathbf{x}, \mathbf{y} \in B$, the entry $(\mathbf{x}, \mathbf{y})$ of the $|B|\times|B|$ Gram matrix $\Psi_B \Psi_B^H$ is given by
$$(\Psi_B \Psi_B^H)_{\mathbf{x}, \mathbf{y}} = \sum_{\chi_{\mathbf{a}} \in \mathcal{F}} \chi_{\mathbf{a}}(\mathbf{x}) \overline{\chi_{\mathbf{a}}(\mathbf{y})}.$$
Since $A$ contains $B$, this implies that $\Psi_B \Psi_B^H$ is exactly the principal sub-matrix of $\Psi_A \Psi_A^H$ obtained by restricting both row and column indices to the subset $B$. By Lemma \ref{lem:rayleigh}, the largest eigenvalue of any Gram matrix $M$ is given by
$$\lambda_{\max}(M) = \max\{\mathbf{u}^H M \mathbf{u}:\|\mathbf{u}\|=1\}.$$
For any unit vector $\mathbf{u} \in \mathbb{C}^{|B|}$, we can define an augmented vector $\tilde{\mathbf{u}}=(\mathbf{u}, \mathbf{0}) \in \mathbb{C}^{|A|}$.
Evidently, $\|\tilde{\mathbf{u}}\| = \|\mathbf{u}\| = 1$. The corresponding quadratic form then satisfies
$$\mathbf{u}^H \Psi_B \Psi_B^H \mathbf{u}= \tilde{\mathbf{u}}^H \Psi_A \Psi_A^H \tilde{\mathbf{u}}.$$
Since $\lambda_{\max}(\Psi_A \Psi_A^H) = \max\{ \mathbf{w}^H \Psi_A \Psi_A^H \mathbf{w}:\|\mathbf{w}\|=1\}$, we have
$$\mathbf{u}^H \Psi_B \Psi_B^H \mathbf{u} = \tilde{\mathbf{u}}^H \Psi_A \Psi_A^H \tilde{\mathbf{u}} \le \lambda_{\max}(\Psi_A \Psi_A^H).$$
Taking the maximum over all unit vectors $\mathbf{u} \in \mathbb{C}^{|B|}$ on the left-hand side then yields
$$\lambda_{\max}(\Psi_B \Psi_B^H) \le \lambda_{\max}(\Psi_A \Psi_A^H) = |A|.$$
By Lemma \ref{lem:rayleigh} with the all-ones vector $\mathbf{1}_{|B|}$, we obtain
\begin{align*}\label{prime}
|A| \ge \lambda_{\max}(\Psi_B \Psi_B^H) &\ge \frac{\mathbf{1}_{|B|}^H (\Psi_B \Psi_B^H) \mathbf{1}_{|B|}}{\mathbf{1}_{|B|}^H \mathbf{1}_{|B|}}\\
&=\frac{(\Psi_B^H \mathbf{1}_{|B|})^H(\Psi_B^H\mathbf{1}_{|B|})}{\mathbf{1}_{|B|}^H \mathbf{1}_{|B|}}\\
&=\frac{\| \Psi_B^H \mathbf{1}_B \|^2}{|B|}\\
\end{align*}
\begin{align*}
&=\frac{\sum\limits_{\chi_{\mathbf{a}} \in \mathcal{F}}\left|\sum\limits_{\mathbf{x} \in B} \overline{\chi_\mathbf{a}(\mathbf{x})}\right|^2}{|B|}\\
&=\frac{\sum\limits_{\chi_{\mathbf{a}} \in \mathcal{F}}\left|\overline{\sum\limits_{\mathbf{x} \in B} \chi_\mathbf{a}(\mathbf{x})}\right|^2}{|B|}\\
&=\frac{1}{{|B|}} \sum_{\chi_{\mathbf{a}} \in \mathcal{F}} \left|\sum_{\mathbf{x} \in B} \chi_{\mathbf{a}}(\mathbf{x})\right|^2. \tag{3}
\end{align*}
Let $\chi_{\mathbf{a}} = \chi_{(a_1, \dots, a_k)} \in \mathcal{F}$, where $a_i$ denotes the $i$-th component of the index vector $\mathbf{a}$, corresponding to the $i$-th factor of the array. Since $B$ is an orthogonal array of strength $t$, its projection onto the coordinates $I = \operatorname{supp}(\mathbf{a})$ is uniform for every $|I| \le t$. Thus each row with $I$-pattern in $B_I$ appears exactly $|B| / \prod_{i \in I} r_i$ times. Consequently, we have

\begin{align*}\label{eq:Sa_square}
\sum\limits_{\mathbf{x} \in B} \chi_{\mathbf{a}}(\mathbf{x})
&= \sum_{\mathbf x \in B} \prod_{i=1}^k \omega_{s_i}^{a_i x_i} \\
&= \sum_{\mathbf x \in B} \prod_{i \in I} \omega_{s_i}^{a_i x_i}
   \qquad (\text{since } \omega_{s_i}^{0}=1) \\
&= \sum_{\mathbf y_I \in H}
   \sum_{\substack{\mathbf x \in B \\ x_i = y_i, \ i \in I}}
   \prod_{i \in I} \omega_{s_i}^{a_i y_i} \\
&= \sum_{\mathbf y_I \in H}
   \left( \prod\limits_{i \in I} \omega_{s_i}^{a_i y_i} \right)
   \cdot \frac{|B|}{\prod\limits_{i \in I} r_i} \\
&= \frac{|B|}{\prod\limits_{i \in I} r_i}
   \sum_{\mathbf y_I \in H}
   \prod_{i \in I} \omega_{s_i}^{a_i y_i} \\
&= \frac{|B|}{\prod\limits_{i \in I} r_i}
   \prod_{i \in I}
   \left( \sum_{y=0}^{r_i-1} \omega_{s_i}^{a_i y} \right),
\end{align*}
where $\omega_{s_i} = e^{2\pi i / s_i}$. Thus
\begin{align*}
\left|\sum\limits_{\mathbf{x} \in B} \chi_{\mathbf{a}}(\mathbf{x})\right|^2
= \frac{|B|^2}{\prod\limits_{i\in I} r_i^2}
\prod_{i\in I} \left|\sum_{x=0}^{r_i-1} \omega_{s_i}^{a_i x}\right|^2\tag{4}.
\end{align*}

Substituting Equation (\ref{eq:Sa_square}) into  Equation (\ref{prime}), and by Lemma \ref{lem:combined_univariate}, we obtain

\[
\begin{aligned}
|A|&\geq\frac{1}{|B|}\sum_{\chi_{\mathbf{a}} \in \mathcal{F}} \left| \sum_{\mathbf{x} \in B} \chi_{\mathbf{a}}(\mathbf{x}) \right|^2
= |B| \sum_{\chi_{\mathbf{a}} \in \mathcal{F}}
   \frac{1}{\prod\limits_{i \in \operatorname{supp}(\mathbf{a})} r_i^2}
   \prod\limits_{i \in \operatorname{supp}(\mathbf{a})} \left| \sum_{x=0}^{r_i-1} \omega_{s_i}^{a_i x} \right|^2\\
&=|B|\sum_{\substack{I \subseteq \{1,\dots,k\} \\  \chi_{\mathbf{a}} \in \mathcal{F} \text{ with } \operatorname{supp}(\mathbf{a}) = I}}
\frac{1}{\prod\limits_{i \in I} r_i^2}
\sum_{\substack{a_i \in \{1,\dots,s_i-1\}\\ i \in I}}
\prod_{i \in I} \left| \sum_{x=0}^{r_i-1} \omega_{s_i}^{a_i x} \right|^2 \\
&=|B|\sum_{\substack{I \subseteq \{1,\dots,k\} \\  \chi_{\mathbf{a}} \in \mathcal{F} \text{ with } \operatorname{supp}(\mathbf{a}) = I}}
\frac{1}{\prod\limits_{i \in I} r_i^2}
\prod_{i \in I} \left( \sum_{a_i=1}^{s_i-1} \left| \sum_{x=0}^{r_i-1} \omega_{s_i}^{a_i x} \right|^2 \right) \\
&=|B|\sum_{\substack{I \subseteq \{1,\dots,k\} \\  \chi_{\mathbf{a}} \in \mathcal{F} \text{ with } \operatorname{supp}(\mathbf{a}) = I}}
\frac{1}{\prod\limits_{i \in I} r_i^2}
\prod_{i \in I} r_i(s_i - r_i) \\
&=
|B|\sum_{\substack{I \subseteq \{1,\dots,k\} \\  \chi_{\mathbf{a}} \in \mathcal{F} \text{ with } \operatorname{supp}(\mathbf{a}) = I}}
\prod_{i \in I} \frac{s_i - r_i}{r_i} \\
&=
|B|\sum_{\substack{I \subseteq \{1,\dots,k\} \\ \exists \, \chi_{\mathbf{a}} \in \mathcal{F} \text{ with } \operatorname{supp}(\mathbf{a}) = I}}
\prod_{i \in I} \left( \frac{s_i}{r_i} - 1 \right).
\end{aligned}\]
We now evaluate $$\sum\limits_{\substack{I \subseteq \{1,\dots,k\} \\  \chi_{\mathbf{a}} \in \mathcal{F} \text{ with } \operatorname{supp}(\mathbf{a}) = I}}
\prod_{i \in I} \left( \frac{s_i}{r_i} - 1 \right)$$ over the restricted character set $\mathcal{F}$ by splitting the analysis into two cases based on the parity of $t$.

\noindent\textbf{Case 1: $t$ is even.} Note that
$\mathcal{F} = \{\chi_{\mathbf{a}} : |\operatorname{supp}(\mathbf{a})| \le \frac{t}{2}\},
$ and the admissible support sets are precisely those with $|I| \le \frac{t}{2}$. Hence, we have
\[\sum_{\substack{I \subseteq \{1,\dots,k\} \\ \chi_{\mathbf{a}} \in \mathcal{F} \text{ with } \operatorname{supp}(\mathbf{a}) = I}}
\prod_{i \in I} \left( \frac{s_i}{r_i} - 1 \right)
= \sum_{i=0}^{\frac{t}{2}} \sum_{\substack{I \subseteq \{1,\dots,k\} \\ |I|=i}} \prod_{j \in I} \left( \frac{s_j}{r_j} - 1 \right).\]

\noindent\textbf{Case 2: $t$ is odd.} Note that $\mathcal{F} =\mathcal{S'}\cup\mathcal{U}$, where
$\mathcal{S'} = \{ \chi_{\mathbf{a}} \in \widehat{G} : |\operatorname{supp}(\mathbf{a})| \le \left\lfloor \frac{t}{2}\right\rfloor \}, \
\mathcal{U} = \{ \chi_{\mathbf{a}} : |\operatorname{supp}(\mathbf{a})| = \left\lceil \frac{t}{2}\right\rceil \text{ and } 1 \in \operatorname{supp}(\mathbf{a}) \}, $ and
the admissible support sets fall into two disjoint classes.

(i) For $|I| \le \left\lfloor \frac{t}{2}\right\rfloor$. Their contribution is
\[\sum_{i=0}^{\left\lfloor \frac{t}{2}\right\rfloor} \sum_{\substack{I \subseteq \{1,\dots,k\} \\ |I|=i}} \prod_{j \in I} \left( \frac{s_j}{r_j} - 1 \right). \]

(ii) For $|I| = \left\lceil \frac{t}{2}\right\rceil$ and $1 \in I$. Write $I = \{1\} \cup J$, where $J \subseteq \{2,\dots,k\}$ and $|J| = \left\lfloor \frac{t}{2}\right\rfloor$. By the assumptions, we have $\frac{s_1}{r_1}=\max\{\frac{s_i}{r_i}:1\le i\le k\}$. The maximum value is
\[
\left( \frac{s_1}{r_1} - 1 \right)
\sum_{\substack{J \subseteq \{2,\dots,k\} \\ |J| = \left\lfloor \frac{t}{2}\right\rfloor}} \prod_{j \in J} \left( \frac{s_j}{r_j} - 1 \right).\]

Combining the two classes yields
{\small\[
\sum_{\substack{I \subseteq \{1,\dots,k\} \\ \chi_{\mathbf{a}} \in \mathcal{F} \text{ with } \operatorname{supp}(\mathbf{a}) = I}}
\prod_{i \in I} \left( \frac{s_i}{r_i} - 1 \right)
=\left[
\sum_{i=0}^{\lfloor t/2 \rfloor} \sum_{\substack{I \subseteq \{1,\dots,k\} \\ |I|=i}} \prod_{j \in I} \left( \frac{s_j}{r_j} - 1 \right)
+
\left( \frac{s_1}{r_1} - 1 \right)
\sum_{\substack{I \subseteq \{2,\dots,k\} \\ |I|=\lfloor t/2 \rfloor}} \prod_{j \in I} \left( \frac{s_j}{r_j} - 1 \right)
\right].\]}
This completes the proof.
\end{proof}

A nested orthogonal array is \emph{optimal} if it meets the bound (2) of Theorem \ref{lem:eigenvalues}.
In the symmetrical case where $s_i=s$  and $r_i=r$  for all $i$, (2) of Theorem \ref{lem:eigenvalues} is consistent with the bound given by Mukerjee, Qian, Wu in \cite{MQW}.

\begin{lem}[\!\cite{MQW}]
If there exists an NOA\(((N, M), k, (s, r), t)\) for \( 2\leq t \leq k \), then
\[
N \ge
\begin{cases}
\displaystyle M \sum_{j=0}^{\frac{t}{2}} \binom{k}{j} (\frac{s}{r} - 1)^j, & \text{if } t \ \text{is even } , \\[12pt]
\displaystyle M \left( \sum_{j=0}^{\lfloor\frac{t}{2}\rfloor} \binom{k}{j} (\frac{s}{r} - 1)^j
\;+\;
\binom{k-1}{\lfloor\frac{t}{2}\rfloor} (\frac{s}{r} - 1)^{\lfloor\frac{t}{2}\rfloor+1} \right), & \text{if } t \ \text{is odd }.
\end{cases}
\]
\end{lem}

Lin et al. \cite{Lin} recently derived the following bound for NOAs via the \emph{variance method}.

\begin{lem}[\!\cite{Lin}]\label{Lin}
If there exists an NOA$((N, M),k, (s_1s_2 \dots s_k, r_1r_2\dots r_k), 2)$, then
\begin{equation*}\label{eq:lin_bound}
N \ge M \left( 1 + \frac{\left( \sum\limits_{j=1}^k (1 - \frac{r_j}{s_j}) \right)^2}{\sum\limits_{j=1}^k \frac{r_j}{s_j} (1 - \frac{r_j}{s_j})} \right).
\end{equation*}
\end{lem}
\begin{lem}[\!\cite{Lin}]\label{Lin1}
If there exists an \(\mathrm{NOA}((N, M), k, (s_1s_2 \cdots s_k, r_1r_2\cdots r_k), t)\), where \(s_i \ge r_i\) with strict inequality for at least \(t-1\) \(i\)'s, then
\[N \ge \max_{\{j_1, \dots, j_{t-2}\} \subseteq \{1, \dots, k\}}
\left\{
\frac{M s_{j_1} \cdots s_{j_{t-2}}}{r_{j_1} \cdots r_{j_{t-2}}}
\left[
1 + \frac{\sum\limits_{l \in \{1, \dots, k\} \setminus \{j_1, \dots, j_{t-2}\}} (1 - r_l / s_l)^2}{\sum\limits_{l \in \{1, \dots, k\} \setminus \{j_1, \dots, j_{t-2}\}} (1 - r_l / s_l) r_l / s_l}\right]\right\}.\]

\noindent
For \(1 \le a \le t-2\), if there is only one set \(\{p_1, \dots, p_a\}\) such that \(s_{p_h} > r_{p_h}\), then
\[N \ge
\begin{cases}
\max \{X, Y + \frac{s_{p_1} \cdots s_{p_a}}{r_{p_1} \cdots r_{p_a}} M\}, & \text{if } \frac{N}{s_{p_1} \cdots s_{p_a}} \neq \frac{M}{r_{p_1} \cdots r_{p_a}}, \\[1.2em]
\max \{X, Y\}, & \text{if } \frac{N}{s_{p_1} \cdots s_{p_a}} = \frac{M}{r_{p_1} \cdots r_{p_a}},
\end{cases}\]
where
\[X = \max_{\{j_1, \dots, j_{t-2}\} \supseteq \{p_1, \dots, p_a\}}
\left\{
\frac{M s_{j_1} \cdots s_{j_{t-2}}}{r_{j_1} \cdots r_{j_{t-2}}}
\left[
1 + \frac{\sum\limits_{l \in \{1, \dots, k\} \setminus \{j_1, \dots, j_{t-2}\}} (1 - r_l / s_l)^2}{\sum\limits_{l \in \{1, \dots, k\} \setminus \{j_1, \dots, j_{t-2}\}} (1 - r_l / s_l) r_l / s_l}
\right]
\right\},
\]
\[
Y = \max_{\{j_1, \dots, j_{t-2-a}\} \cap \{p_1, \dots, p_a\}}
\left\{
s_{p_1} \cdots s_{p_a} s_{j_1} \cdots s_{j_{t-2-a}}
\left[
1 + \sum_{f \in \{1, \dots, k\} \setminus \{p_1, \dots, p_a, j_1, \dots, j_{t-2-a}\}} (s_f - 1)
\right]
\right\}.
\]
\end{lem}

\begin{rmk}
Applying the Cauchy--Schwarz inequality
\[
\left( \sum_{j=1}^k a_j b_j \right)^2
\le \left( \sum_{j=1}^k a_j^2 \right) \left( \sum_{j=1}^k b_j^2 \right)
\]
with $a_j = \sqrt{q_j(1-q_j)}$ and $b_j = \sqrt{(1-q_j)/q_j}$  yields
\[
\left( \sum_{j=1}^k (1 - q_j) \right)^2
\le \left( \sum_{j=1}^k q_j(1-q_j) \right) \left( \sum_{j=1}^k \frac{1-q_j}{q_j} \right).
\]
Then
\[\frac{\left( \sum\limits_{j=1}^k (1 - q_j) \right)^2}{\sum\limits_{j=1}^k q_j (1 - q_j)} \le \sum_{j=1}^k \left( \frac{1}{q_j} - 1 \right).\]
Let $q_j = \frac{r_j }{s_j}$, we have
\[
\frac{\left( \sum\limits_{j=1}^k (1 - \frac{r_j }{s_j}) \right)^2}{\sum\limits_{j=1}^k \frac{r_j }{s_j} (1 - \frac{r_j }{s_j})} \le \sum_{j=1}^k \left( \frac{s_j }{r_j} - 1 \right).
\]
Thus
\[M \left( 1 + \frac{\left( \sum\limits_{j=1}^k (1 - \frac{r_j }{s_j}) \right)^2}{\sum\limits_{j=1}^k \frac{r_j }{s_j} (1 - \frac{r_j }{s_j})}\right)\le
M \left( 1 + \sum\limits_{j=1}^k \left( \frac{s_j}{r_j} - 1 \right) \right).
\]
Clearly, the expression on the left is given by Lemma~\ref{Lin}, and the expression on the right can be obtained from Theorem~\ref{lem:eigenvalues} with $t=2$.  Note that for $t \geq3$, Lin et al. \cite{Lin} employed the deleting column technique to reduce the problem to $t=2$ setting.
Therefore, our bound exhibits substantial improvements over Lin's bound, and the resulting formulas are both concise and elegant.
\end{rmk}

\begin{exam}
Consider the NOA$\bigl((31104,1536),26,(2^{15}3^{11},2^{26}),4\bigr)$ given in Example 1 of \cite{Pang2024-SPL}. Here $k=26,\ t=4,\ |A|=31104,\ |B|=1536$, and
\[
\frac{s_i}{r_i} =
\begin{cases}
1, & 1 \le i \le 15, \\[0.3em]
\dfrac{3}{2}, & 16 \le i \le 26.
\end{cases}
\]

We first apply Lemma~\ref{Lin1}. Since $t=4$, we select $t-2=2$ indices $\{j_1, j_2\}$. Let $m \in \{0,1,2\}$ be the number of chosen indices with $s_i/r_i = 3/2$. Then
\begin{center}
$
\frac{s_{j_1}s_{j_2}}{r_{j_1}r_{j_2}} = \left(\frac{3}{2}\right)^m
$
and
$
1-\frac{r_i}{s_i}=
\begin{cases}
0,&1 \le i \le 15,\\[0.3em]
\dfrac13,&16 \le i \le 26.
\end{cases}
$
\end{center}
For a choice with $m$ ternary columns selected, the remaining number of ternary columns is $11-m$, hence
\[
\frac{\sum_{l\notin\{j_1,j_2\}} (1-r_l/s_l)^2}{\sum_{l\notin\{j_1,j_2\}} (1-r_l/s_l)(r_l/s_l)}
= \frac{(11-m)/9}{2(11-m)/9} = \frac{1}{2}.
\]
Maximizing over $m=2$, Lemma~\ref{Lin1} gives
\[
|A| \ge 1536 \times \left(\frac{3}{2}\right)^2 \times \left(1 + \frac{1}{2}\right) = 5184.
\]
Next, applying the Theorem \ref{lem:eigenvalues}, since $t=4$ is even,
\begin{align*}
|A|&\ge |B|\sum_{i=0}^{2}
\sum_{\substack{I\subseteq\{1,\ldots,26\}\\ |I|=i}}
\prod_{j\in I}\left(\frac{s_j}{r_j}-1\right)\\
&=1536\cdot\left(1+\binom{11}{1}\frac12+\binom{11}{2}\frac14\right)\\
&=1536\cdot\frac{81}{4}\\&=31104.
\end{align*}
In summary,
\[
|A|=31104=L_{Thm4.5}>L_{Lin}=5184.
\]
\end{exam}

The following two constructions are provided to show that the lower bounds derived in Theorem \ref{lem:eigenvalues} are indeed attainable and optimal.

\begin{lem}[\!\cite{Brouwer2006}]\label{M} If there exists an OA$(N,a_1a_2\dots a_k,t)$, then there exists an OA$(nN,(na_1)a_2\\ \dots a_k,t)$.
\end{lem}
\begin{thm}\label{uM}
Let $n\in \mathbb{Z}_+$ and $t\geq 2$. If there is an OA$(M, s_1 s_2 \cdots s_k, t)$, then there is an optimal
NOA\(\big((nM, M), k, ((ns_1) s_2 \dots s_k, s_1 s_2 \dots s_k), t\big).
\)
\end{thm}
\begin{proof}
By assumption, there exists an OA$(M, s_1 s_2 \cdots s_k, t)$. Applying Lemma \ref{M} yields an OA$(nM, (ns_1) s_2 \cdots s_k, t)$, which clearly contains the original OA as a subarray and therefore forms an NOA$((nM, M), k, ((ns_1)s_2 \dots s_k, s_1 s_2\dots s_k), t)$.

We establish the optimality of the constructed NOA by showing that it attains the lower bound derived in Theorem~\ref{lem:eigenvalues} exactly. For the NOA$\big((nM, M), k, ((ns_1)s_2\cdots s_k,\ s_1s_2\cdots s_k), t\big)$, we have
$$\frac{ns_1}{s_1}=n, \ \ \ \frac{s_i}{s_i}=1 \ \ \text{for} \ \ i\ge 2.$$
We now evaluate the bound in Theorem~\ref{lem:eigenvalues} with these parameters. Since $\frac{s_i}{s_i}-1=0$ for all $i\ge 2$, only the index $i=1$ contributes to the sums in the lower bound.

\textbf{Case 1: $t$ is even.} The lower bound therefore simplifies to
\[|A| \ge M\sum_{|I|\le t/2}\prod_{i\in I}\left(\frac{s_i}{r_i}-1\right)
= M\bigl(1+(n-1)\bigr)=nM.\]
Note that $\sum\limits_{|I|=\emptyset}\prod\limits_{i\in I}\left(\frac{s_i}{r_i}-1\right)=1$.

\textbf{Case 2: $t$ is odd.} We have $\lfloor t/2\rfloor\ge 1$. Consequently, the lower bound reduces to
\begin{align*}
|A| &\ge M\left(
\sum_{|I|\le \lfloor t/2\rfloor}\prod_{i\in I}\left(\frac{s_i}{r_i}-1\right)
+(u-1)\sum_{\substack{|I|=\lfloor t/2\rfloor\\1\notin I}}
\prod_{i\in I}\left(\frac{s_i}{r_i}-1\right)
\right)\\
&=M(1+(n-1)+0)=nM.
\end{align*}

In both cases, the lower bound reduces to $|A| \ge nM$. Since the proposed construction yields an array with exactly $nM$ runs, the bound is attained, thereby proving the optimality of the NOA.
\end{proof}

The following is the juxtaposition construction of orthogonal arrays.

\begin{lem}[\!\cite{Chen}]\label{chen}
If there exists an OA$(N_1, a_1a_2  \cdots a_k, t) $ and an OA$(N_2,b_1a_2  \cdots a_k, t) $ satisfying $ \frac {N_1} {a_1}= \frac {N_2} {b_1} $, then there exists an OA$(N_1+N_2, (a_1+b_1) a_2  \cdots a_k, t)$.
\end{lem}

The following simple yet powerful result follows directly from Lemma \ref{chen} and Theorem \ref{lem:eigenvalues}.
\begin{thm} \label{thm:noa_construction}
If there exists an OA$(N_1, a_1s_2 \cdots s_k, t)$ and an OA$(N_2, b_1s_2 \cdots s_k, t)$ satisfying  $\frac{N_1}{a_1}=\frac{N_2}{b_1}$, then there exists an optimal
NOA$((N_1 + N_2, N_1),k, ((a_1 + b_1) s_2\cdots s_k, b_1s_2 \cdots s_k), t).$

\end{thm}
\begin{proof}
Let $A_1$ and $A_2$ be an OA$(N_1, a_1s_2 \cdots s_k, t)$  and an OA$(N_2, b_1s_2 \cdots s_k, t)$, respectively.
Let $B = \begin{pmatrix} A_1 \\ A_2 \end{pmatrix}$. It is easy to verify that $B$ is an NOA$((N_1 + N_2, N_1),k, ((a_1 + b_1)s_2 \cdots s_k, b_1s_2 \cdots s_k), t)$ by Lemma \ref{chen}.

We show that the construction achieves the spectral bound in Theorem~\ref{lem:eigenvalues}. In our setting, for the first factor, the numbers of levels in the larger and smaller orthogonal arrays are $s_1 = a_1 + b_1$ and $r_1 = a_1$, respectively. For all other factors $i \in \{2, \dots, k\}$, we have $s_i = r_i$, which implies $\frac{s_i}{r_i} - 1 = 0$.

For both $t = 2u$ and $t = 2u+1$, the lower bound in Theorem~\ref{lem:eigenvalues} reduces to the same expression, since the enhancement term vanishes for odd strength: for any $1 \notin I$ with $|I| = u \ge 1$, we have $\prod\limits_{i \in I} \bigl( \frac{s_i}{r_i} - 1 \bigr) = 0$. Hence,
\[
|A| \ge M \sum_{|I| \le u} \prod_{i \in I} \left(\frac{s_i}{r_i} - 1\right)= M \left[1+( \frac{s_1}{r_1}-1)\right] =  M \frac{s_1}{r_1}.
\]
Substituting $M = N_1$, $s_1 = a_1 + b_1$, $r_1 = a_1$ and using the construction condition $\frac{N_1}{a_1} = \frac{N_2}{b_1}$ (i.e., $N_2 = N_1 \frac{b_1}{a_1}$), we obtain
\[
|A| \ge N_1 + N_1 \frac{b_1}{a_1} = N_1 + N_2.
\]
Thus, the total run size of the constructed nested orthogonal array exactly attains the theoretical lower bound, confirming the optimality of the construction for any strength $t$ under the single-factor asymmetry assumption.
\end{proof}

\section{Conclusions}

In this paper, we first present a novel proof of Rao's bound. The improvement of the corresponding upper bound remains a worthy direction for further exploration. Second, we generalize the new lower bound for nested orthogonal arrays originally established by Mukerjee, Qian, and Wu \cite{MQW} propose two simple construction methods that achieve this bound. Improving the upper and lower bounds for NOAs and constructing optimal NOAs will be the focus of our future research.

\newpage
\vskip 12pt
\noindent
\textbf{Declaration of competing interest}

The authors declare that they have no known competing financial interests or personal relationships that could have appeared to influence the work reported in this paper.

\vskip 12pt
\noindent
\textbf{ Data availability}

No data was used for the research described in the article.

\vskip 12pt

\noindent
\textbf{Funding}
The work was funded by Science Research Project of Hebei Education Department (Grant NO. JCZX2026032), the Natural Science Foundation of Hebei Province (Grant NO. A2025205023). This work was
also supported by the Natural
Science Foundation of Henan Province (Grant No. 262300421846)  and the National Natural Science Foundation of China (Grant No. 62372157).

\end{document}